\documentclass[11 pt]{amsart}

\usepackage{amsthm}
\usepackage{amssymb}
\usepackage{latexsym}
\usepackage{multicol}
\usepackage{verbatim,enumerate}
\usepackage{accents}
\usepackage[usenames]{color}
\usepackage{hyperref}
\usepackage{hyperref}
\usepackage{amsmath, amscd}
\usepackage{soul}
\usepackage{tikz}
\usetikzlibrary{matrix,arrows,decorations.pathmorphing}
\usetikzlibrary{positioning}

\advance\textwidth by 1.2in \advance\oddsidemargin by -.6in \advance\evensidemargin by -.6in

\theoremstyle{plain}
\newtheorem{prop}{Proposition}
\newtheorem{thm}{Theorem}
\newtheorem{lem}{Lemma}
\newtheorem{cor}{Corollary}
\newtheorem{obv}{Observation}

\theoremstyle{definition}
\newtheorem{example}{Example}
\newtheorem{defn}{Definition}
\newtheorem{rem}{Remark}

\theoremstyle{remark}

\newcommand{\lie}[1]{\mathfrak{#1}}

\newcommand\bc{\mathbb C}
\newcommand\bn{\mathbb N}
\newcommand\bz{\mathbb Z}

\DeclareRobustCommand\longtwoheadrightarrow
     {\relbar\joinrel\twoheadrightarrow}

\def\a{\alpha}
\def\l{\lambda}
\def\d{\delta}
\def\b{\beta}
\def\m{\mu}

\newcounter{cnt}
 \makeatletter
\def\mydggeometry{\makeatletter\dg@YGRID=1\dg@XGRID=20\unitlength=0.003pt\makeatother}
\makeatother \theoremstyle{remark}

\numberwithin{equation}{section}
\makeatletter
\def\section{\def\@secnumfont{\mdseries}\@startsection{section}{1}%
  \z@{.7\linespacing\@plus\linespacing}{.5\linespacing}%
  {\normalfont\scshape\centering}}
\def\subsection{\def\@secnumfont{\bfseries}\@startsection{subsection}{2}%
  {\parindent}{.5\linespacing\@plus.7\linespacing}{-.5em}%
  {\normalfont\bfseries}}
\makeatother

\begin{document}


\title{Lattice of dominant weights of Affine Kac-Moody algebras}

\author{Krishanu Roy}
\thanks{The author was partially supported by Raman-Charpak Fellowship (2019) and ISF Grant no. 1221/17 .}
\address{Bar Ilan University, Ramat Gan, Israel}
\email{krishanur@imsc.res.in}
\subjclass[2010]{17B10, 17B67}
\keywords{dominant weights, basic cell, covering relations}

\maketitle

\begin{abstract}
The dual space of the Cartan subalgebra in a Kac-Moody algebra has a partial ordering defined by the rule that two elements are related if and only if their difference is a non-negative or non-positive integer linear combination of simple roots. In this paper, we study the subposet formed by dominant weights in affine Kac-Moody algebras. We give a more explicit description of the covering relations in this poset. We also study the structure of basic cells in this poset of dominant weights for untwisted affine Kac-Moody algebras of type $A$. 
\end{abstract}

\section{Introduction}

Let $\lie{g}$ denote a Kac-Moody algebra with Cartan subalgebra $\lie{h}$ and set of roots $\Phi$. The dual space $\lie{h}^*$ has a partial ordering defined as follows: $\l> \m$ if and only if $\l-\m\in\bn\Phi^+$ where $\Phi^+$ denote the set of positive roots. The subposet of dominant weights is particularly of great interest because of its connection to the representation theory of integrable highest weight module over $\lie{g}$. To give one example of this, consider a finite, affine or strictly hyperbolic Lie algebra $\lie{g}$ and an integrable highest weight module $L(\l)$ over $\lie{g}$. The integrability of $L(\l)$ implies that $\l$ is a dominant weight. The module $L(\l)$ is $\lie{h}$-diagonalizable and has a weight space decomposition $L(\l)=\oplus_{\m\in P(\l)}L(\l)_\m$ where $P(\l)$ is the set of weights. The set $P(\l)$ is completely determined up to Weyl conjugacy by the dominant weights contained in it, i.e. $P(\l)=W(P(\l)\cap \Lambda^+)$ where $W$ denote the Weyl group and $\Lambda^+$ denote the set of dominant weights corresponding to $\lie{g}$. The set $P(\l)\cap \Lambda^+$ can be described by the partial ordering on $\Lambda^+$ as follows: $\m\in P(\l)\cap \Lambda^+$ if and only if $\m\leq \l.$

\medskip 

In fact, the original motivation of Stembridge to study this partial order in \cite{stem} was to compute the weight multiplicities in the finite dimensional highest weight modules over finite type Kac-Moody algebras by using Freudenthal's algorithm. It is useful to know the explicit descriptions of the covering relations in the poset of dominant weights towards calculating this multiplicity. The covering relations in this poset of dominant weights for finite type Lie algebras were described explicitly in \cite{stem} (See also \cite{bry}). Here we generalize the arguments of \cite{stem} to describe the partial order in the affine case (e.g. Theorem \ref{1sttheo} is the affine version of Theorem 2.6 of \cite{stem}).

\medskip

Another motivation to study this partial order for the affine Kac-Moody algebras arose while studying the atomic decomposition of the characters of integrable highest weight modules over affine Kac-Moody algebras. In \cite{lus}, Lusztig defined a $t$-analogue $K_{\l,\m}(t)$ of the weight multiplicity of $\m$ in the irreducible highest weight representation with the highest weight $\l$ over finite type Kac-Moody algebras generalizing the Kostka-Foulkes polynomials which are just $t$-analogue of Kostka-Foulkes numbers. In \cite{las}, Lascoux proved the decomposition of Kostka-Foulkes polynomials into {\em{atomic polynomials}} (See also \cite{shi}). Later in \cite{ced}, the atomic decomposition was formulated for all finite type Kac-Moody algebras and was proved for a large number of cases while conjecturing that it holds more generally. The explicit description of the covering relations and the basic cell structure in the poset of dominant weights played an important role in the proof.  

\medskip

In this paper, we study the poset of dominant weights for all affine type Lie algebras. We prove several results about the structure of the poset ($\Lambda^+,\leqslant$) for affine Lie algebras, some of which can be realized as generalizations of the poset structure of dominant weights in complex semisimple Lie algebras. For example, we prove that connected components of ($\Lambda^+,\leqslant$) are lattices (Lemma \ref{lattice}). We also prove that $\l$ covers $\m$ in this ordering only if $\l-\m$ is either the canonical imaginary root $\d$ or it belongs to a distinguished subset of positive real roots (Theorem \ref{1sttheo}). 

\medskip

 The paper is organized as follows. In section 2, we set up the notations, recall some basic facts and deduce some basic results about the dominant weights of affine Lie algebras. In section 3, we present a detailed analysis of the covering relations in ($\Lambda^+,\leqslant$). In particular, we prove that a dominant weight $\l$ is covered by $\l+\d$ only if $\l$ is a fundamental weight with only one possible exception (Lemma \ref{fun}). In section 4, we analyze the basic cell structure of ($\Lambda^+,\leqslant$) for the untwisted affine type $A$. We prove that the basic cells are of shape either diamond or pentagon (Theorem \ref{cell}). The basic cell structures for other affine Lie algebras can be studied as well using Theorem \ref{2ndthe}, and we hope to present the results elsewhere.

\medskip

{\bf{Acknowledgements:}} The author is extremely grateful to C\'edric Lecouvey for suggesting to look into the covering relations and basic cell structure of the poset of dominant weights of affine Lie algebras, and for many illuminating discussions. The author also thanks him for his detailed comments on this manuscript. The author acknowledges hospitality and excellent working conditions at Institut Denis Poisson (UMR CNRS 7013), University of Tours, where part of the work was done.

\section{Preliminaries}

We denote the set of complex numbers by $\bc$ and the set of integers, non-negative integers, and positive integers by $\bz$, $\bz_{\geq 0}$, and $\bz_{> 0}$ respectively. We refer to \cite{Kac} for the general theory of affine Lie algebras and affine root systems. 

\medskip

Through out this article, $A=(a_{ij})_{i,j\geq 0}^n$ will denote an affine generalized Cartan matrix (GCM) of order $n+1$, $\lie{g}(A)$ its corresponding Kac-Moody algebra with set of roots $\Phi$, simple roots $\a_0,\a_1,\cdots\a_n$, co-roots $\a_0^\vee,\a_1^\vee,\cdots\a_n^\vee$, Dynkin diagram $S(A)$ and Cartan subalgebra $\lie{h}$. So $\a_j(a_i^{\vee})=a_{ij}$ is non-positive for $i \neq j$ and equal to 2 for $i=j$. Let $\mathring{A}$ be the matrix obtained from $A$ by deleting the 0-th row and column. We will call $\mathring{A}$ to be the finite part of $A$.

\medskip

 Let ( , ) denote the normalized invariant form [\cite{Kac}, Page 81] on $\lie{g}(A)$. Notice that its restriction on the root lattice induces a positive semi-definite bilinear form [\cite{Kac}, Proposition 4.7]. In a simply laced diagram all roots are considered both short and long roots. Let $$\Lambda=\{\lambda\in\lie{h}^*: \lambda(\a_i^\vee)\in\bz \text{ for all } 0\leq i\leq n\}$$
denote the set of integral weights, where $\lie{h}$ is the Cartan subalgebra of  $\lie{g}(A)$. 
Let $$\Lambda^+=\{\lambda\in\lie{h}^*: \lambda(\a_i^\vee)\in\bz_{\geq 0} \text{ for all } 0\leq i\leq n\}$$ denote the set of dominant integral weights. There is a partial order $\leqslant$ on $\Lambda$ given by $\mu\leq\lambda$ if and only if $\l-\mu$ is a non negative integral sum of simple roots. Our goal is to study the covering relations of this partial order restricted to the set of dominant weights.

\medskip

Let $c$ denote the canonical central element and $\delta$ denote the canonical imaginary root of $\lie{g}(A)$. Then $\delta=\sum_{i=0}^n a_i\a_i $ and $c=\sum_{i=0}^n a_i^{\vee}\a_i^{\vee}$ where $a_0,a_1,\cdots,a_n$ (resp. $a_0^{\vee},a_1^{\vee},\cdots,a_n^{\vee}$) be the positive numerical labels of $S(A)$ (resp. $S(A^T)$) in Table Aff. in [\cite{Kac}, Page (54,55)]. Let  $$\Lambda^+_m= \{\lambda\in\Lambda^+ : \lambda(c)=m\}$$ denote the set of dominant weights of level m. Here, we make our first observation which follows immediately from the fact that $\a_i(c)=0$ for all $0\leq i \leq n$.
\vspace{1em}

\medskip

\begin{obv}\label{level}
If $\mu\leq\lambda$ in ($\Lambda^+,\leqslant$), then $\lambda(c)=\mu(c)$ i.e. all dominants weights of a connected component of ($\Lambda^+,\leqslant$) lie in $\Lambda^+_m$ for some fixed m.
\end{obv}

\medskip 

\begin{example}
Let $A$ be the GCM of type $A_n^{(1)}$. The Dynkin diagram $S(A)$ corresponding to $A$ is cyclic (for $n\geq 2$). The finite part $\mathring A$ of $A$ is of type $A_n$. The canonical central element is just the sum of simple co-roots i.e.  $c=\sum_{i=0}^n \a_i^{\vee}$. Similarly the canonical imaginary root is the sum of simple roots i.e. $\delta=\sum_{i=0}^n \a_i $. 
\end{example}

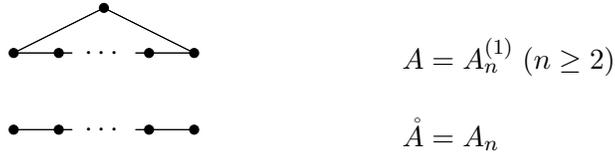
\begin{figure}[h]
\begin{tabular}{rp{2cm}l}
 
\begin{tikzpicture}[scale=0.6]
    \filldraw[] (0,0) -- +(1.3,0);
      \filldraw[] (2.7,0) -- +(1.3,0);
     \filldraw[] (0,0) -- +(2,1);
        \filldraw[] (4,0) -- +(-2,1);
        \foreach \x in {0,1}
                 {    \filldraw[] (\x,0)  circle (.1cm); }
                 \foreach \x in {3,4}
             {    \filldraw[] (\x,0)  circle (.1cm); }
\filldraw[] (2,1) circle (.1cm);
             \draw (2,0) node[] {$\cdots$};
\end{tikzpicture}

 && $A={A}_n^{(1)} \; (n \geq 2)$\\ \\
\begin{tikzpicture}[scale=0.6]
    \filldraw[] (0,0) -- +(1.3,0);
      \filldraw[] (2.7,0) -- +(1.3,0);
        \foreach \x in {0,1}
                 {    \filldraw[] (\x,0)  circle (.1cm); }
                 \foreach \x in {3,4}
             {    \filldraw[] (\x,0)  circle (.1cm); }
             \draw (2,0) node[] {$\cdots$};
\end{tikzpicture}
  && $\mathring A ={A}_n $\\ \\
\end{tabular}
\caption{Dynkin diagrams}
\end{figure}

\medskip

Here, we recall some basic facts and deduce some basic results related to affine root systems. 

\medskip

\begin{defn}
A weight $\l $ is called a fundamental weight corresponding to a vertex $i$ in $S(A)$ if $\l(\a_j^{\vee})=\d_{ji}$ for all $j$.
\end{defn}

\medskip

We define an element $\omega_0\in\lie{h}^*$ determined uniquely by 
$$\omega_0(\a_0^{\vee})=1 \text{ and } \omega_0(h)=0 \text{ for } h\in\lie{h}\setminus \bc \a_0^{\vee }.$$

Notice that $\omega_0$ is a fundamental weight corresponding to the 0-th vertex. Recall that we have dim $\lie{h}^*=2(n+1)-rank (A)=n+2$, as $rank (A)=n$. The following proposition gives us a basis of $\lie{h}^*$. 

\begin{prop}[\cite{car}, Proposition 17.4]
 $\omega_0,\a_0,\a_1,\cdots,\a_n$ form a basis of $\lie{h}^*$.
\end{prop}

\medskip

Given an element $\l\in\Lambda^+_m$, we look at the coefficient of $\omega_0$ when $\l$ is expressed in terms of the above basis. The following lemma shows that the coefficient of $\omega_0$ is the same as the level of $\l$. This lemma also simultaneously justifies Observation \ref{level}.

\begin{lem}[\cite{car}, Lemma 20.1]
Let $\l\in\Lambda^+_m$. Then $\l=m\omega_0+\sum\limits_{i=0}^n\l_i\a_i$ for some $\l_i$. 
\end{lem}

\medskip

Now we look at the poset ($\Lambda^+,\leqslant$) which is of our primary interest. First notice that Observation \ref{level} proves that ($\Lambda^+,\leqslant$) has infinitely many connected components. The fact that $\l\in\Lambda^+$ if and only if $\l+n\d\in\Lambda^+$ for all $n\in\bz$, proves that none of these connected components has either a maximum or a minimum element. Nevertheless, the next lemma shows that each connected component of ($\Lambda^+,\leqslant$) has a lattice structure.

\begin{lem}\label{lattice}
Each connected component of ($\Lambda^+,\leqslant$) is a lattice.
\end{lem}

\begin{proof}
Let $\l$ and $\m$ belong to a connected component of ($\Lambda^+,\leqslant$) say, $(\Lambda^+_a,\leqslant)$. Therefore, we have $\l-\m=\sum\limits_{i=0}^n k_i\a_i$ for some $k_i\in\bz$. Recall that $\d=\sum\limits_{i=0}^n a_i\a_i$ for some $a_i\in\bz_{>0}$. Now choose $n=max[\{k_i:0\leq i\leq n\}\cup\{0\}]$. Then we obtain $\l,\m\geq \l-n\d.$ So each pair in $(\Lambda^+_a,\leqslant)$ has a lower bound. Similarly, we get $\l,\m\leq \m+n\d.$ So each pair in $(\Lambda^+_a,\leqslant)$ also has a upper bound.

\medskip 

As $\l$ and $\m$ belong to a connected component of ($\Lambda^+,\leqslant$), recall that by Observation \ref{level} we have $\l(c)=\m(c)= m$. Suppose $\l=m\omega_0+\sum\limits_{i=0}^n\l_i\a_i$ and $\m=m\omega_0+\sum\limits_{i=0}^n\m_i\a_i $. Now let us define $$\l\wedge\m=m\omega_0+\sum\limits_{i=0}^n min\{\l_i,\m_i\}\a_i.$$
By arguments similar to Lemma 1.2 and subsection 1.1 in \cite{stem}, we obtain $\l \wedge \m \in (\Lambda^+_a,\leqslant)$. Clearly $\l \wedge \m$ is the greatest lower bound of $\l$ and $\m$. Now to prove$(\Lambda^+_a,\leqslant)$ is a lattice, all we need to show is that for any two elements $\l,\m$ in $(\Lambda^+_a,\leqslant)$, their lowest upper bound exists. Let $U(\l,\m)$ be the set of all upper bounds of $\l$ and $\m$. Observe that $U(\l,\m)$ is non-empty. Let $\eta=m\omega_0+\sum\limits_{i=0}^n\eta_i\a_i$ be an arbitrary element in $U(\l,\m)$. Note that $\eta_i$ is bounded below by both $\l_i$ and $\m_i$. Choose $\nu_i\in  U(\l,\m)$ such that it minimizes the coefficient of $\a_i$ in $U(\l,\m)$ for each $0\leq i \leq n$. Now clearly $\nu_0\wedge \nu_1\wedge\cdots\wedge\nu_n$ is the lowest upper bound of $\l$ and $\m$. This finishes the proof.
\end{proof}

\section{Covering relation}
From here on we will deal only with dominant weights of positive level i.e. we will consider only those dominant weights $\mu$ such that there exists some $i$ for which $\mu(\a_i^{\vee})\neq 0$.  We recall here an important Lemma and a Proposition from \cite{stem}, which are crucial for the study of the covering relations on ($\Lambda^+,\leqslant$). For the sake of completeness, we will include the proof of Lemma here.

\medskip

Let $K$ be a subdiagram of $S(A)$. For $\b=\sum_{i=0}^n k_i\a_i$, we define $Supp (\b)=\{{i:k_i\neq 0}\}$ and $\b|_{K}:=\sum_{i\in K}k_i\a_i$. If $K$ is a proper connected subdiagram then by [\cite{Kac}, Proposition 4.7], $K$ must be of finite type. We denote the short highest root corresponding to the subdiagram $K$ by $\a_K$.
\begin{prop}[Proposition 2.1, \cite{stem}]\label{mainprop}
Let $\Psi$ be an irreducible finite root system and $\lambda$ be a dominant weight. If $\lambda(\neq 0)$ is a non-negative integral sum of simple roots then $\lambda-\a_{\Psi}$ is also a non-negative integral sum of simple roots where $\a_{\Psi}$ is the short highest root of $\Psi$. 
\end{prop}

\begin{lem}[Lemma 2.5, \cite{stem}]\label{main}
Suppose $\mu<\mu +\b$ in ($\Lambda^+,\leqslant$), $I=\text{Supp }\b$, $J=\{i\in I : \mu(\a_i^{\vee})=0\}$ and $K\subset I$ be a proper connected subdiagram of $S(A)$.  
\begin{enumerate}
\item[(1)] If $\b|_{K}(\a_i^{\vee})\geq 0$ for all $i\in K-J$, then $\beta\geq \a_K$.
\item[(2)] If in addition $(\mu+\a_K)(\a_i^{\vee})\geq 0$ for all $i\in I-K$, then $\mu+\a_K$ is dominant.
\end{enumerate}
\end{lem}

\begin{proof}
 For $i\in J$, $\b(\a_i^{\vee})=(\mu+\b)(\a_i^{\vee})\geq 0$ as $\mu+\b\in\Lambda^+$. So if $i\in K\cap J$, then $$\b|_{K}(\a_i^{\vee})=\b(\a_i^{\vee})-(\b-\b|_{K})(\a_i^{\vee})\geq \b(\a_i^{\vee})\geq 0,$$
 since $i\notin Supp(\b-\b|_{K})$. Combining this with the hypothesis $(1)$, we get $\b|_{K}-\a_K=\sum\limits_{j\in K}m_j\a_j$ for some $m_j\in\bz_{\geq 0}$ by applying Proposition \ref{mainprop} on the irreducible finite root system $\Phi_K$ corresponding to the subdiagram $K$. Hence we get $\b\geq \b|_{K}\geq \a_K$. This proves (1).
 
\medskip
 
For $i\notin I$, $$(\mu+\a_K)(\a_i^{\vee})=[(\mu+\b)-(\b-\a_K)](\a_i^{\vee})\geq 0,$$ since $\mu+\b$ is dominant and $i\notin Supp(\b-\a_{K})$.
For $i\in I\cap K$, $$(\mu+\a_K)(\a_i^{\vee})\geq 0,$$ as $\mu\in\Lambda^+$ and $\a_K$ dominant with respect to $\Phi_K$. This combined with the stated hypothesis proves that $\mu+\a_K\in\Lambda_+$.
\end{proof}

\medskip

Here we recall a general version of Cauchy-Schwarz inequality for a real vector space with a positive semi-definite symmetric bi-linear form. We will use it to establish some basic facts about affine root systems.
\begin{prop}\label{cauchy}
Let V be a real vector space with a positive semi-definite symmetric bi-linear form ( , ). Then for $u,v\in V$, $(u,v)^2\leq |u|^2|v|^2$. If the equality holds then either $|v|^2=0$ or $|u-\frac{(u,v)}{|v|^2}v|=0.$
\end{prop}
\begin{lem}\label{rem}
Consider $\a$ and $\b$ two real roots of $\lie{g}(A)$. Then  $\a(\b^{\vee})\b(\a^{\vee})\leq 4$.
\end{lem}
\begin{proof}
By Proposition \ref{cauchy}, $(\a,\b)^2\leq |\a|^2|\b|^2$. So, $\a(\b^{\vee})\b(\a^{\vee})=\frac{2(\a,\b)}{|\b|^2}\frac{2(\a,\b)}{|\a|^2}\leq 4$.
\end{proof}
\begin{lem}\label{1stgen}
Let $\a\in\Phi$ be a real root and $\alpha_i$ be a simple root such that $\a(\a_i^{\vee})\a_i(\a^{\vee})= 4$ and $\a_i\notin Supp(\a)$. Then $\a-\frac{(\a,\a_i)}{|\a_i|^2}\a_i=-\frac{(\a,\a_i)}{a_i|\a_i|^2}\delta$.
\end{lem}

\begin{proof}
Since the bi-linear form ( , ) is positive semi-definite [Proposition 4.7, \cite{Kac}], by applying Proposition \ref{cauchy}, $$(\a,\a_i)^2\leq |\a|^2 |\a_i|^2.$$

By hypothesis $\a(\a_i^{\vee})\a_i(\a^{\vee})= 4$. therefore we get $\frac{2(\a,\a_i)}{|\a|^2}\frac{2(\a,\a_i)}{|\a_i|^2}=4$. Hence the above equality holds. So,
$$|\a-\frac{(\a,\a_i)}{|\a_i|^2}\a_i|^2=0.$$
As the radical of the bi-linear form ( , ) is 1-dimensional and it is generated by $\delta$, we get $\a-\frac{(\a,\a_i)}{|\a_i|^2}\a_i=n\delta$ for some $n$. Since $\a_i\notin Supp(\a)$, $n=-\frac{(\a,\a_i)}{a_i|\a_i|^2}$ where $a_i$ is the coefficient of $\a_i$ in the expansion of $\delta$. Hence the result.
\end{proof}

\begin{cor}
Let $\a$ and $\a_i$ be as in Lemma \ref{1stgen}. If $\frac{(\a,\a_i)}{|\a_i|^2}=-1$, then we have $\a+\a_i=\d$.
\end{cor}

\medskip

\begin{lem}\label{affdel}
Let $\b\in\bz\Phi$ be an element of the affine root lattice corresponding to the affine GCM $A$ with the property that $\b(\a_i^{\vee})\geq 0$ for all $0\leq i \leq n$. Then $\b=n\delta$ for some $n\in\bz$.
\end{lem}
\begin{proof}
Let $\b=\sum\limits_{i=0}^nm_i\a_i$ for some $m_i\in\bz$. Since $\b(\a_i^{\vee})\geq 0$ for all $0\leq i \leq n$, we have $$Am\geq 0$$ where $m=[m_0,m_1,\cdots,m_n]^T$. By [Theorem 4.3, \cite{Kac}], $Am=0$ and $\b=n\delta$ for some $n\in\bz$.
\end{proof} 

\medskip

Now we state our first main theorem which describes the covering relations in the poset ($\Lambda^+,\leqslant$) for any affine GCM $A$.

\begin{thm}\label{1sttheo}
If $\lambda$ covers $\mu$ in ($\Lambda^+_m,\leqslant$) for some $m\in \mathbb{Z}_{>0}$. Then one of the following is true:
\begin{enumerate}
\item[(1)] $\lambda-\mu=\a_K$ for some proper connected subdiagram K of $S(A)$ 
\item[(2)] $\lambda-\mu=\delta$ where $\delta$ is the canonical imaginary root of $\lie{g}(A)$.
\item[(3)] $A=D_4^{(3)}$ or $A=G_2^{(1)}$, $\mu$ is not a fundamental weight and $\lambda-\mu=\a_1+\a_2$, where $\a_1$ and $\a_2$ are the simple roots generating a root system of type $G_2$.
\item[(4)] $A=G_2^{(1)}$,  $\mu$ is a multiple of the fundamental weight corresponding to the unique short simple root and $\lambda-\mu=\a_1+\a_2+\a_3$, where $\a_1$, $\a_2$ and $\a_3$ are the simple roots of $A$.
\end{enumerate} 
\end{thm}

\begin{proof}
This theorem is clearly true for the case $A=A_1^{(1)}$. Since $\delta(A_1^{(1)})$ is just the sum of two simple roots, $\lambda-\mu$ is either a simple root or $\lambda-\mu=\delta$. So for the proof, let us assume $A\neq A_1^{(1)}$.

\medskip

Let $\b=\lambda-\mu$, $I=\text{Supp }\b$ and $J=\{i\in I:\mu(\a_i^{\vee})=0\}$. If there exists a proper connected subdiagram $K$ in $I$ satisfying the conditions of Lemma \ref{main}, then $\mu+\a_K$ is dominant. Since $\lambda$ covers $\mu$. we get $\lambda-\mu=\a_K$. Otherwise, we claim that either $\lambda-\mu=\delta$ or it falls in the two exceptional cases i.e. (3) and (4).

\medskip

{\bf{Case 1:}} Suppose $J=\emptyset$. Let us choose $K=\{i\}$ for some $i\in I$ which is short related to $I$. The condition (1) of Lemma \ref{main} is trivial. Since $J=\emptyset$, for $j\in I\backslash \{i\}$, $\mu(\a_j^{\vee})\geq 1$ and   we get $\a_i(\a_j^{\vee})\geq -1$ as $A\neq A_1^{(1)}$. Hence $$(\mu+\a_i)(\a_j^{\vee})\geq 0.$$
Therefore, the condition (2) of Lemma \ref{main} is also satisfied. Hence we get $\lambda-\mu=\a_K$, where $K$ contains the only vertex $\a_i$.

\medskip

{\bf{Case 2:}} Suppose $J\neq\emptyset$ and let $K$ be a connected component of $J$ containing a short root relative to $J$. So $\a_K$ is short relative to $J$.  Since $m\in\mathbb{Z}_{>0}$, there exists at least on $i\in S(A)$ such that $\mu(\a_i^{\vee})\neq 0$. So, $J$ is a proper subdiagram of $S(A)$ and hence so is $K$.

\medskip

{\bf{Subcase 2.1:}} Suppose $(\mu+\a_K)(\a_i^{\vee})\geq 0$ for all $i\in I-J$. The condition (1) of lemma \ref{main} is vacuous here. For $i\in J-K$, $\a_K(\a_i^{\vee})=0$ since $K$ is a connected component of $J$ and hence $(\mu+\a_K)(\a_i^{\vee})=0$. So the condition (2) of lemma \ref{main} is also satisfied.

\medskip

{\bf{Subcase 2.2:}} Now let us assume there exist $i\in I-J$, such that $(\mu+\a_K)(\a_i^{\vee})< 0$. Since $i\notin J$, $\mu(\a_i^{\vee})\in\mathbb{Z}_{>0}$. So 
$$\a_K(\a_i^{\vee})\leq -2.$$

By Lemma \ref{rem}, $\a_K(\a_i^{\vee})$ could be -2, -3 or -4. We will now study each cases one by one. Let $L$ be the connected component of $J\cup\{i\}$ containing $\a_i$.

\medskip

{\bf{Subcase 2.2.1:}} Suppose $\a_K(\a_i^{\vee})= -2$. Lemma \ref{rem} shows that the root length of $\a_K$ can not be strictly smaller than $\a_i$. If $|\a_K|^2= |\a_i|^2$, then  $a_K$ and $\a_i$ satisfy the conditions of Lemma \ref{1stgen}. So, $\a_K+\a_i=\delta$ and hence $$\mu+\delta=\mu+\a_K+\a_i\leq \mu+\b.$$ 
This implies $\lambda-\mu=\delta$. 

\medskip

 Suppose that the root length of $\a_K$ is strictly bigger than $\a_i$. This means that the root length of $\a_i$ is strictly smaller than that of the short root of $J$.  Then $\a_i(\a_K^{\vee})$ must be $-1$. Hence $|\a_K|^2=2|\a_i|^2$. If all roots in $J$ have same length, then $\a_i$ is the only simple root in $L$ which is shorter than all other roots. Hence $\a_j(\a_i^{\vee})\in 2\bz \text { \;for all\;  } j\in L$. If $J$ has roots of two different lengths, then $\Phi$ has roots of 3 different lengths. So, $\Phi$ is of type $A_{2l}^{(2)}$ and $\a_i$ is the unique short simple root. Hence, $\a_j(\a_i^{\vee})\in 2\bz \text { \;for all\;  } j$.

So in any case, $$\gamma(\a_i^{\vee})\in 2\bz \text { \;for all\;  } \gamma\in\bz\Phi_L.$$

\medskip

Since $(\mu+\a_K)(\a_i^{\vee})< 0$, $\a_K(\a_i^{\vee})= -2$ and $i\notin J$, $\mu(\a_i^{\vee})$ must be equal to 1. Therefore, $$\b|_L(\a_i^{\vee})\geq \b(\a_i^{\vee})=\lambda(\a_i^{\vee})-\mu(\a_i^{\vee})\geq -1.$$

As $\b|_L(\a_i^{\vee})$ is even, $\b|_L(\a_i^{\vee})\geq 0$. So the condition (1) of the Lemma \ref{main} is satisfied for $L$. 

\medskip 

If $L=S(A)$ then by Lemma \ref{affdel}, $\b=n\delta$ for some $n\in\bz_{\geq 0}$. Hence $$\mu+\delta\leq \mu+\b.$$ 
This implies $\lambda-\mu=\delta$. 

\medskip

Now suppose $L$ is a proper subdiagram of $S(A)$. If $L$ contains a short root relative to $S(A)$, then $\a_L$ is short and hence $\a_L(\a_j^{\vee})\geq -1$ for all $j\in I\setminus J$ . If $L$ does not contain a short simple root relative to $S(A)$, then $A= A_{2l}^{(2)}$, $K=J$ just contains the unique long simple root and $\a_i$ is the intermediate simple root connected to the long simple root. Hence $\a_L(\a_j^{\vee})\geq -1$ for all $j\in I\setminus J$  for $l\geq 3$ since the short simple root in not connected to $L$. As $L$ is a Connected component of $J\cup \{i\}$, we have $\a_L(\a_j^{\vee})= 0$ for $j\in J\setminus L$. So for $j\in I\setminus L$, $$(\mu+\a_L)(\a_j^{\vee})\geq 0 .$$
Therefore, the condition (2) of Lemma \ref{main} is also satisfied for $L$.

\medskip

Now let's look at the case $A=A_4^2$. Let $\a_1,\a_2$ and $\a_3$ be short, intermediate and long simple roots of $A$ respectively. So, we have $K=J=\{\a_3\}$, $i=2$ and $L=\{\a_2,\a_3\}$. If $1\notin I$, then  $\a_L(\a_j^{\vee})\geq -1$ for all $j\in I\setminus J$ and $(\mu+\a_L)$ is dominant just as before. Now let us assume $1\in I$. since $\a_3\in J$, $\m$ is of the form $m_1\omega_1+m_2\omega_2+m\d$. If $m_2>0$, then $\m+\a_1$ is dominant. So we have $m_2=0$. If $m_1>1$, then  $\m+\a_2+\a_3$ is dominant. So $m_1$ must be 1. A direct check shows that for $\m=\omega_1+m\d$, $\b$ must be equal to $\d$.

\medskip

{\bf{Subcase 2.2.2:}} Suppose $\a_K(\a_i^{\vee})= -3$. Then $\Phi$ must be of type $G_2^{(1)}$ or $D_4^{(3)}$ by Table Aff. in [\cite{Kac}, Page (54,55)]. First suppose $\Phi$ is of type 
$G_2^{(1)}$ and let $\a_i,\a_j,\a_k$ be the simple roots. Since $\a_K(\a_i^{\vee})= -3$, $\a_i$ must be the unique short simple root. Hence $K=J$, as $J$ is connected and $K$ is a connected component of $J$. Let $\a_j$ be the vertex connected to $\a_i$. So, $j\in K$ as $\a_K(\a_i^{\vee})\neq 0$. We claim $\mu+\a_K+\a_i$ is dominant. $$(\mu+\a_K+\a_i)(\a_i^{\vee})\geq 1-3+2=0,$$ $$(\mu+\a_K+\a_i)(\a_j^{\vee})\geq 0+1-1=0 .$$ If $k\in K$, then $\mu$ is a multiple of the fundamental weight corresponding to the unique short simple root $\a_i$ and $(\mu+\a_K+\a_i)(\a_k^{\vee})\geq 0+1+0=1$.

\medskip
 
If $k\notin K$, then $J=K=\{j\}$ and hence $\mu$ is not a fundamental weight. In this case $\a_K(\a_k^{\vee})= -1$ and $\mu(\a_k^{\vee})\geq 1$. So, $(\mu+\a_K+\a_i)(\a_k^{\vee})\geq 0$. Hence $\lambda-\mu=\a_K+\a_i$.

\medskip

Suppose $\Phi$ is of type $D_4^{(3)}$. Since $\a_K(\a_i^{\vee})= -3$, $\a_K$ is a long root with respect to $\Phi$. As $\a_K$ is short corresponding to $J$, $J$ can contain only the unique long simple root say $\a_k$. Therefore, we have $K=J$ and $\mu$ is not a fundamental weight. Let $\a_i$ be the short simple root connected to $\a_k$ and let $\a_j$ be the other short simple root. Then $$(\mu+\a_k+\a_i)(\a_i^{\vee})\geq 1-3+2=0,$$ $$(\mu+\a_k+\a_i)(\a_j^{\vee})\geq 1+0-1=0 ,$$
$$(\mu+\a_k+\a_i)(\a_k^{\vee})= 0+2-1=1.$$ 

Hence $\mu+\a_k+\a_i$ is dominant and hence $\lambda-\mu=\a_k+\a_i$.

\medskip

{\bf{Subcase 2.2.3:}} Suppose $\a_K(\a_i^{\vee})= -4$. This means $\a_K$ and $\a_i$ generate a root system of type $A_2^{(2)}$. So $\Phi$ is of type $A_{2l}^{(2)}$. 

\medskip

Since $\a_K(\a_i^{\vee})= -4$, $\a_K$ is a long root and $\a_i$ is the short simple root with respect to $\Phi$. As $\a_K$ is short corresponding to $J$, $J$ can contain only the unique long simple root say $\a_k$ and hence $\a_K=\a_k$. The fact $\a_k(\a_i^{\vee})\neq 0$ implies $\Phi$ must be of type $A_2^{(2)}$. Since $(\mu+\a_K)(\a_i^{\vee})< 0$, $i\notin J$  and $\a_K(\a_i^{\vee})= -4$, we get $\mu(\a_i^{\vee})\in \{1,2,3\}$. Therefore
$$(\mu+\a_K+\a_i)(\a_i^{\vee})=\mu(\a_i^{\vee})-4+2,$$ $$(\mu+\a_K+\a_i)(\a_K^{\vee})=0+2-1=1.$$

If $\mu(\a_i^{\vee})\geq 2$, then $\mu+\a_K+\a_i$ is dominant. Since $\mu+\a_K+\a_i\leq \mu+\b$, we have $\lambda-\mu=\a_K+\a_i$. If $\mu(\a_i^{\vee})=1$, then $\mu+n\a_K+\a_i$ is not dominant for any $n\in\bz_{\geq 0}$. But $\mu+\a_K+2\a_i$ is dominant. By Lemma \ref{1stgen} $\a_K+2\a_i=\delta$. So, $\lambda-\mu=\delta$.

\medskip




This finishes the proof.

\end{proof}

\subsection{}

Let $CR(A)$ denote the set of all roots appearing in the statement of Theorem \ref{1sttheo}. It is clear from theorem \ref{1sttheo} that $\lambda$ covers $\mu$ in ($\Lambda^+_m,\leqslant$) implies $\lambda-\mu\in CR(A)$. The following results strengthen this theorem.

\begin{defn} A vertex $i$ of the Dynkin diagram $S(A)$ is called special if $L:=S(A)\setminus \{i\}$ is connected and $\delta=\a_L+\a_i.$
\end{defn}

\medskip 

\begin{example}
Let us take $A$ to be the untwisted affine GCM of type $A_{n}^{(1)}$. The canonical imaginary root for $A$ is the sum of all simple roots i.e. $\delta=\sum_{j\in S(A)}\a_j$ Let $i$ be any vertex of the Dynkin diagram $S(A)$. Then $L:=S(A)\setminus \{i\}$ is connected and is of type $A_n$. So, we have $\a_L=\sum_{j\neq i}\a_j$, and hence  $\delta=\a_L+\a_i.$ Therefore, every vertex in  $A_{n}^{(1)}$ is a special vertex.
\end{example}

\medskip

\begin{lem}\label{spver}
A simple root corresponding to a special vertex is always a short root. 
\end{lem} 
\begin{proof}
Let $i$ be a special vertex in $S(A)$. Then we have $\a_L=\delta-\a_i$ and hence $|\a_L|^2=|\a_i|^2$. Suppose $\a_i$ is not a short root. Then the subdiagram $L$ must contain a short root say, $\a_j$. Then we have $|\a_L|^2=|\a_j|^2<|\a_i|^2$ which is a contradiction. Hence $\a_i$ must be a short root in $S(A).$
\end{proof}

\medskip

\begin{lem}\label{fun}
Let $\mu$ covers $\mu-\delta$ in ($\Lambda^+_m,\leqslant$). Then one of the following is true:
\begin{enumerate}
\item $\mu$ is a fundamental weight corresponding to a special vertex of $S(A)$.
\item $S(A)$ is a non triply laced Dynkin diagram containing a unique short simple root and $\mu$ is a fundamental weight corresponding to that unique short simple root.
\item $A=D_{n+1}^{(2)}$ and $\mu=\omega_0+\omega_n$.
\item $A=A_{1}^{(1)}$ and $\mu=\omega_0+\omega_1$.
\end{enumerate}
 
\begin{proof}
First notice that $\mu(\a_i^{\vee})\leq 1$ for all $i$. Otherwise, $\mu-\a_i$ would be dominant. This would contradict the fact $\mu$ covers $\mu-\delta$.

\medskip

Now let $\mu(\a_i^{\vee})=\mu(\a_j^{\vee})= 1$ for some $i\neq j$. Choose the smallest connected subdiagram K of $S(A)$ with endpoints $i$ and $j$. For  $q\notin K$, $$(\mu-\sum_{p \in K}\a_p)(\a_q^{\vee})\geq 0.$$

For $q\in K-\{i,j\}$, $\sum_{p \in K}\a_p(\a_q^{\vee})\leq 0$ as $\a_q$ has at least 2 neighbours. Therefore, $$(\mu-\sum_{p \in K}\a_p)(\a_q^{\vee})\geq\mu(\a_q^{\vee})\geq 0.$$ And $(\sum_{p \in K}\a_p)(\a_i^{\vee})\leq 1$, $(\sum_{p \in K}\a_p)(\a_j^{\vee})\leq 1$ as both $\a_i$ and $\a_j$ have 1 neighbour. So, we have $$(\mu-\sum_{p \in K}\a_p)(\a_i^{\vee})\geq 0 \text{  and  }(\mu-\sum_{p \in K}\a_p)(\a_j^{\vee})\geq 0.$$
This contradicts the fact $\mu$ covers $\mu-\delta$  unless $\sum_{p \in K}\a_p=\delta$. By Table Aff. in [Page (54,55), \cite{Kac}] , this is possible when either $A=D_{n+1}^{(2)}$ and $\mu=\omega_0+\omega_n$ or $A=A_{1}^{(1)}$ and $\mu=\omega_0+\omega_1$. Thus, except these two special cases, $\mu$ is a fundamental weight corresponding to a vertex, say $i$.

\medskip

Let $L$ be a connected component of $S(A)\setminus \{i\}$. Note that for $j\in L$ $$\delta|_L(\a_j^{\vee})=\delta(\a_j^{\vee})-(\delta-\delta|_L)(\a_j^{\vee})\geq 0.$$ Therefore, by Proposition \ref{mainprop}, $\a_L\leq \delta|_L$ and hence $\a_L\leq \delta$. As $\mu$ covers $\mu-\delta$, $\mu-\delta+\a_L$ is not dominant. For $j\notin L$ and $j\neq i$, we have
$$(\mu-\delta+\a_L)(\a_j^{\vee})=\mu(\a_j^{\vee})\geq 0,$$
since $L$ is a connected component of $S(A)\setminus \{i\}$. For $j\in L$, we have
$$(\mu-\delta+\a_L)(\a_j^{\vee})\geq 0,$$
as $\a_L$ is dominant on $L$. Hence we get $$(\mu-\delta+\a_L)(\a_i^{\vee})<0.$$
Therefore, $\a_L(\a_i^{\vee})\leq -2$. 

\medskip

If $A$ is simply laced, then by Lemma \ref{1stgen}, $\a_L+\a_i=\delta$. Hence $i$ is a special vertex.

\medskip

Now assume $A$ is non simply laced. We claim that in this case $\a_i$ must be a short root of $S(A)$. Suppose $\a_i$ is not short and L be the connected component of $S(A)\setminus \{i\}$ containing a short root. Then $\a_i(a_L^{\vee})<\a_L(a_i^{\vee})\leq -2$ contradicting the Lemma \ref{rem}.

\medskip

Suppose $\a_i$ is not the unique short simple root in $S(A)$ and let $\a_j$ be a short root different from $\a_i$. Let $L$ be the connected component of $S(A)\setminus \{\a_i\}$ containing $\a_j$. Then as before we get $\a_L(\a_i^{\vee})\leq -2$. Since $\a_L$ and $\a_i$ have same root length, $\a_i(\a_L^{\vee})\leq -2$. Therefore, $\a_L(\a_i^{\vee})=\a_i(\a_L^{\vee})= -2$. Now by Lemma \ref{1stgen}, $\a_L+\a_i=\delta$. 

\medskip

Let $\a_i$ be the unique short simple root in $S(A)$. Then by Theorem \ref{1sttheo}, $A$ can not be $G_2^{(1)}$. Hence the claim.
\end{proof}

The following theorem describes the covering relations in ($\Lambda^+_m,\leqslant$) more explicitly.

\end{lem}
\begin{thm}\label{2ndthe}
If $\mu<\lambda$ in ($\Lambda^+_m,\leqslant$), $I= \text{Supp }(\l-\m)$ and $J=\{i\in I:\mu(\a_i^{\vee})=0\}$, then $\lambda$ covers $\mu$ if and only if I is a connected subdiagram of $S(A)$ and one of the following holds.
\begin{enumerate}
\item[(a)] $\lambda-\mu$ is a simple root.
\item[(b)] $I=J$ is a proper subdiagram of $S(A)$ and $\lambda-\mu=\a_I$.
\item[(c)] $I=J\cup \{i\}$, $\Phi_I$ is of type $B_l$, $\a_i$ is short, $\mu(\a_i^{\vee})=1$ and $\lambda-\mu=\a_I$.
\item[(d)] $I=J\cup \{i\}$, $\Phi_I$ is of type $G_2$, $\a_i$ is short, $\mu(\a_i^{\vee})=\{1,2\}$ and $\lambda-\mu=\sum_{i\in I}\a_i$.
\item[(e)] $I=J\cup \{i\}=S(A)$, $\Phi_I$ is of type $G_2^{(1)}$, $\a_i$ is short, $\mu(\a_i^{\vee})\in\{1,2\}$ and $\lambda-\mu=\sum_{i\in I}\a_i$.
\item[(f)] $\lambda-\mu=\delta$ and $\lambda$, $\mu$ both are fundamental weights corresponding to a special vertex. 
\item[(g)] $\lambda-\mu=\delta$. $S(A$) is a Dynkin diagram which is not triply laced and contains a unique short simple root. $\lambda$, $\mu$ both are fundamental weights corresponding to that unique short simple root. 
\item[(h)] $A=D_{n+1}^{(2)}$, $\lambda=\omega_0+\omega_n$ and $\lambda-\mu=\delta$. 
\item[(i)] $A=A_{1}^{(1)}$, $\lambda=\omega_0+\omega_1$ and $\lambda-\mu=\delta$. 
\end{enumerate}
\end{thm}

\begin{proof}
If $\lambda$ covers $\mu$, then Theorem \ref{1sttheo} and Lemma \ref{fun} imply one of the above-mentioned cases holds. 

\medskip

Now we will prove the converse i.e. each of the cases (a)-(i) gives rise to a covering relation. If $\lambda-\mu$ is a simple root, then this is immediate. For case (b), assume $I=J$ be a proper subdiagram of $S(A)$, $\lambda-\mu=\a_I$ and $\lambda$ does not cover $\mu$. Since $\mu+\a_I$ does not cover $\mu$, by Theorem \ref{1sttheo}, there exists a proper connected subdiagram $K$ of $I$ such that $\mu+\a_K$ covers $\mu$. As $K$ is a proper subdiagram of $I$, there exists $i\in I$ which is connected to $K$. The fact $$\a_K(\a_i^{\vee})<0$$ implies $\mu(\a_i^{\vee})>0$, contradicting the fact $I=J$. Therefore, $\lambda$ covers $\mu$. This proves the case $(b)$.

\smallskip

 For the cases $(c)$ and $(d)$, observe that $\l-\m\in CR(A)$; say $\l-\m=\eta.$ If $\m+\eta$ does not cover $\m$, then there exists some $\xi<\eta$ in $CR(A)$ such that $\m+\xi$ is dominant. But $\xi<\eta$ implies that $\xi$ is not dominant in $\Phi_I$. Therefore, there exists $j\in I$ such that $\xi(\a_j^{\vee})<0$. As $\m+\xi$ is dominant, we get, $\m(\a_j^{\vee})>0$. Hence $j$ must be same as $i$. As $\a_i$ is the only short root in $\Phi_I$, $\xi(\a_j^{\vee})<0$ implies that $\xi(\a_j^{\vee})=-2$ (in type B) or $\xi(\a_j^{\vee})=-3$ (in type $G_2$). So for dominance of $\mu+\xi$ it requires $\mu(\a_i^{\vee})\geq 2$ and $\mu(\a_i^{\vee})\geq 3$ respectively, a contradiction. This proves the cases $(c)$ and $(d)$. 

\medskip

For the cases $(e)$, notice that $\mu$ is either a fundamental weight or twice a fundamental weight corresponding to the short simple root $\a_i$. A direct check shows that both cases give a covering relation.

\medskip

Let $\lambda$ be a fundamental weight corresponding to the special vertex $i$, i.e. $\lambda(\a_i^{\vee})=1$ and $\lambda(\a_j^{\vee})=0$ for $j\neq i$. We claim $\lambda+\delta$ covers $\lambda$.

\medskip

Assume $\lambda+\sum_{j=0}^n k_j\a_j$ covers $\lambda$ in ($\Lambda^+_m,\leqslant$) and $k_l=0$ for some $l\neq i$. As $$(\lambda+\sum_{j=0}^n k_j\a_j)(\a_l^{\vee})=\sum_{j\neq l} k_j\a_j(\a_l^{\vee})\geq 0, $$
$k_j=0$ if $j$ is connected to $l$. Since $L:=S(A)\setminus\{i\}$ is connected, $k_j=0$ for all $j\neq i$. But $\lambda+k_i\a_i$ is not dominant, therefore $k_l\neq 0$ for all $l\neq i$. Now suppose $k_i=0$. Then by Theorem \ref{1sttheo}, $\lambda+\a_L$ covers $\lambda$. Hence, $$(\lambda+\a_L)(\a_i^{\vee})\geq 0 \textbf{ }\Rightarrow \a_L(\a_i^{\vee})\geq -1.$$
Since $i\notin L$ and $i$ is connected to $L$, $\a_L(\a_i^{\vee})\leq -1$ and therefore $\a_L(\a_i^{\vee})= -1$. This means $|\a_L+\a_i|^2\neq 0$, contradicting the fact that $\a_L+\a_i=\delta$. So, $k_i\neq 0$. Again by applying Theorem \ref{1sttheo}, we get $\lambda+\delta$ covers $\lambda$. This proves the case $(f)$.
 
\medskip

For case (g), suppose $S(A)$ is a Dynkin diagram which is not triply laced, and contains a unique short simple root say, $\a_i$. Suppose $\lambda$ is a fundamental weight corresponding to $\a_i$. Assume $\lambda+\sum_{j=0}^n k_j\a_j$ covers $\lambda$ in ($\Lambda^+_m,\leqslant$). Note that $\a_j(\a_i^{\vee})$ is even for all $j\neq i$. We have, $$(\lambda+\sum_{j=0}^n k_j\a_j)(\a_i^{\vee})=1+\sum_{j\neq i} k_j\a_j(\a_i^{\vee})+2k_i\geq 0. $$

Therefore we get, $k_i\neq 0$. And hence $k_j\neq 0$ for all $j$ by similar arguments as before. Now by Theorem \ref{1sttheo}, we get $\lambda+\delta$ covers $\lambda$. This proves the case $(g)$.

\medskip

For the case $(h)$, we have $A=D_{n+1}^{(2)}$ and $\lambda=\omega_0+\omega_n$. Suppose $\lambda+\sum_{j=0}^n k_j\a_j$ covers $\lambda$ in ($\Lambda^+_m,\leqslant$). Assume $k_j\neq 0$ for some $0\leq j\leq n$. Since for any $0< l< n$, $$(\lambda+\sum_{j=0}^n k_j\a_j)(\a_l^{\vee})=\sum_{j\neq l} k_j\a_j(\a_l^{\vee})+2k_l\geq 0,$$
 we get $k_l\neq 0$ for any $l$ connected to $j$. Therefore $k_l\neq 0$ for all $0< l< n$. Since $\a_1(\a_0^{\vee})=\a_{n-1}(\a_n^{\vee})=-2$, $k_0,k_n\neq 0$. As $\delta(D_{n+1}^{(2)})$ is just the sum of all simple roots, we get that $\lambda+\delta$ covers $\lambda$. An easy check also proves the case $(i)$. This finishes the proof.

\end{proof}

\medskip

\begin{rem}
One of the important observations from Theorem 2.6 in \cite{stem} was that the covering relations in the poset of dominant weights of finite type Kac-Moody algebras are given by positive roots i.e. if $\l$ covers $\m$ in the poset of dominant weights then $\l-\m$ is a positive root. The above theorem proves that this statement is true even in the case of affine type Kac-Moody algebras. The covering relations in this affine case are given by locally short dominant roots, cannonical imaginary roots and exceptional roots. Following the convention of \cite{stem}, a root of the affine Kac-Moody algebra is called exceptional if it is the sum of simple roots that generate a root system of type $G_2$ or $G_2^{(1)}$.
\end{rem}

\section{Basic cells of the lattice of dominant weights}

If $\lambda$ covers $\mu$, we call $\mu$ a cocover of $\lambda$ and it is denoted by $\lambda \longtwoheadrightarrow \mu $. By Theorem \ref{2ndthe}, covering relations in ($\Lambda^+_m,\leqslant$) for any simply laced Dynkin diagram $S(A)$ are given either by a proper subdiagram $K$ of $S(A)$ or by $\delta$. 
Observe that if $\lambda$ has two distinct cocovers $\mu$ and $\mu'$, then by Theorem \ref{2ndthe} both $\mu$ and $\mu'$ correspond to two different proper subdiagrams of $S(A)$ i.e. none of the cocovers corresponds to $\delta$. 
Now we give an explicit description of the basic cell structure in ($\Lambda^+_m,\leqslant$) for type $A_{n+1}^{(1)}$.
\begin{thm}\label{cell}
If $\mu$ and $\mu'$ are two distinct cocovers of $\lambda$ corresponding to the proper subdiagrams $K$ and $K'$ respectively, then the interval $X=[\mu \wedge \mu',\lambda]$ has one of the following structures:
\begin{enumerate}
\item If $K$ and $K'$ satisfy one of the following conditions:
\begin{itemize}
\item[(a)] $K$ and $K'$ both are singleton diagrams.
\item[(b)] $K\cup K'$ is a disconnected subdiagram of $S(A)$
 \item[(c)] $K\cup K'$ is connected and $K\cap K' \neq \emptyset$, 
\end{itemize}

\noindent
then,

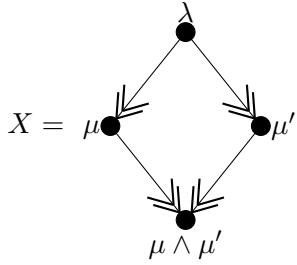
\begin{tikzpicture}[scale=.5]

  \begin{scope}[shift={(55,0)}]
  \filldraw (2,0) node[above] {$\lambda$} circle (.25cm);
  \filldraw (0,-2.5) node[left] {$X=\text{ } \mu$} circle (.25cm);
  \filldraw (4,-2.5) node[right] {$\mu'$} circle (.25cm);
\filldraw (2,-5) node[below] {$\mu \wedge \mu'$} circle (.25cm);
  
  \draw (2,0) -- +(2,-2.5);
  \draw (2,0) -- +(-2,-2.5);
  \draw (0,-2.5) -- +(2,-2.5);
  \draw (4,-2.5) -- +(-2,-2.5);
  
  \draw[thick] (0.32,-2.1) -- +(0.05,.7);
  \draw[thick] (0.32,-2.1) -- +(.7,.2);
\draw[thick] (0.16,-2.3) -- +(0.05,.7);
 \draw[thick] (0.16,-2.3) -- +(.7,.2);
  \draw[thick] (3.68,-2.1) -- +(-0.05,.7);
\draw[thick] (3.68,-2.1) -- +(-.7,.2);
\draw[thick] (3.84,-2.3) -- +(-0.05,.7);
 \draw[thick] (3.84,-2.3) -- +(-.7,.2);
 \draw[thick] (2.32,-4.6) -- +(0.05,.7);
  \draw[thick] (2.32,-4.6) -- +(.7,.2);
\draw[thick] (2.16,-4.8) -- +(0.05,.7);
 \draw[thick] (2.16,-4.8) -- +(.7,.2);
\draw[thick] (1.68,-4.6) -- +(-0.05,.7);
\draw[thick] (1.68,-4.6) -- +(-.7,.2);
\draw[thick] (1.84,-4.8) -- +(-0.05,.7);
 \draw[thick] (1.84,-4.8) -- +(-.7,.2);

\end{scope}
\end{tikzpicture}

\item If $K\cup K'$ is connected, $K\cap K' = \emptyset$, $|K|=1$ and $|K'|>1$  then,

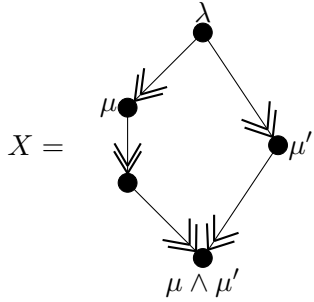
\begin{tikzpicture}[scale=.5]

  \begin{scope}[shift={(55,0)}]
  \filldraw (2,0) node[above] {$\lambda$} circle (.25cm);
  \filldraw (0,-2) node[left] {$\mu$} circle (.25cm);
\filldraw (0,-4) node[left] {$$} circle (.25cm);
  \filldraw (4,-3) node[right] {$\mu'$} circle (.25cm);
\filldraw (2,-6) node[below] {$\mu \wedge \mu'$} circle (.25cm);
  
  \draw (2,0) -- +(-2,-2);
\draw (0,-2) -- node[left] {$X= \hspace{20pt}$}+(0,-2) ;

\draw (0,-4) -- +(2,-2);
\draw (2,0) -- +(2,-3);
\draw (4,-3) -- +(-2,-3);
\draw[thick] (0.35,-1.65) -- +(0.09,.7);
  \draw[thick] (0.35,-1.65) -- +(.7,.2);
\draw[thick] (0.19,-1.82) -- +(0.09,.7);
 \draw[thick] (0.19,-1.82) -- +(.75,.19);
\draw[thick] (0,-3.55) -- +(-0.35,.6);
  \draw[thick] (0,-3.55) -- +(.35,.6);
\draw[thick] (0,-3.75) -- +(-0.36,.59);
  \draw[thick] (0,-3.75) -- +(.36,.59);
\draw[thick] (3.72,-2.58) -- +(-0.69,.26);
 \draw[thick] (3.72,-2.58) -- +(.06,.75);
\draw[thick] (3.86,-2.78) -- +(-.75,.25);
 \draw[thick] (3.86,-2.78) -- +(.08,.72);

\draw[thick] (2.3,-5.55) -- +(0,.7);
  \draw[thick] (2.3,-5.55) -- +(.65,.22);
\draw[thick] (2.14,-5.78) -- +(-0.03,.7);
 \draw[thick] (2.14,-5.78) -- +(.7,.22);
\draw[thick] (1.65,-5.65) -- +(0,.7);
  \draw[thick] (1.65,-5.65) -- +(-.72,.13);
\draw[thick] (1.82,-5.82) -- +(-.75,.12);
 \draw[thick] (1.82,-5.82) -- +(.01,.72);

\end{scope}
\end{tikzpicture}
\item If $K\cup K'$ is connected, $K\cap K' = \emptyset$, $|K|>1$ and $|K'|>1$  then,

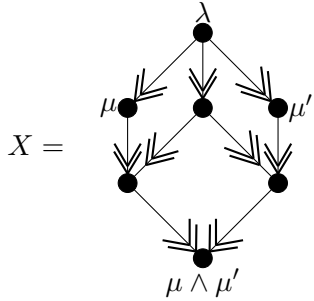
\begin{tikzpicture}[scale=.5]

  \begin{scope}[shift={(55,0)}]
  \filldraw (2,0) node[above] {$\lambda$} circle (.25cm);
  \filldraw (0,-2) node[left] {$\mu$} circle (.25cm);
\filldraw (0,-4) node[left] {$$} circle (.25cm);
 \filldraw (2,-2) node[above] {$$} circle (.25cm);
  \filldraw (4,-2) node[right] {$\mu'$} circle (.25cm);
\filldraw (4,-4) node[right] {$$} circle (.25cm);
\filldraw (2,-6) node[below] {$\mu \wedge \mu'$} circle (.25cm);

  \draw (2,0) -- +(-2,-2);
\draw (2,0) -- +(0,-2);
\draw (2,-2) -- +(-2,-2);
\draw (2,-2) -- +(2,-2);
\draw (0,-2) --node[left] {$X= \hspace{20pt}$}  +(0,-2);
\draw (0,-4) -- +(2,-2);
\draw (2,0) -- +(2,-2);
\draw (4,-2) -- +(0,-2);
\draw (4,-4) -- +(-2,-2);
  
\draw[thick] (0.35,-1.65) -- +(0.09,.7);
  \draw[thick] (0.35,-1.65) -- +(.7,.2);
\draw[thick] (0.19,-1.82) -- +(0.09,.7);
 \draw[thick] (0.19,-1.82) -- +(.75,.19);

\draw[thick] (3.65,-1.65) -- +(-.78,.29);
  \draw[thick] (3.65,-1.65) -- +(-.1,.75);
\draw[thick] (3.81,-1.82) -- +(-.8,.25);
 \draw[thick] (3.81,-1.82) -- +(-.08,.8);

\draw[thick] (0,-3.55) -- +(-0.35,.6);
  \draw[thick] (0,-3.55) -- +(.35,.6);
\draw[thick] (0,-3.75) -- +(-0.36,.59);
  \draw[thick] (0,-3.75) -- +(.36,.59);

\draw[thick] (0.65,-3.35) -- +(0.09,.7);
  \draw[thick] (0.65,-3.35) -- +(.7,.2);
\draw[thick] (0.49,-3.52) -- +(0.09,.7);
 \draw[thick] (0.49,-3.52) -- +(.75,.19);

\draw[thick] (4,-3.55) -- +(-0.35,.6);
  \draw[thick] (4,-3.55) -- +(.35,.6);
\draw[thick] (4,-3.75) -- +(-0.36,.59);
  \draw[thick] (4,-3.75) -- +(.36,.59);

\draw[thick] (2,-1.55) -- +(-0.35,.6);
  \draw[thick] (2,-1.55) -- +(.35,.6);
\draw[thick] (2,-1.75) -- +(-0.36,.59);
  \draw[thick] (2,-1.75) -- +(.36,.59);

\draw[thick] (3.35,-3.35) -- +(-.78,.29);
  \draw[thick] (3.35,-3.35) -- +(-.1,.75);
\draw[thick] (3.51,-3.52) -- +(-.8,.25);
 \draw[thick] (3.51,-3.52) -- +(-.08,.8);

\draw[thick] (1.65,-5.65) -- +(0,.7);
  \draw[thick] (1.65,-5.65) -- +(-.72,.13);
\draw[thick] (1.82,-5.82) -- +(-.75,.12);
 \draw[thick] (1.82,-5.82) -- +(.01,.72);

\draw[thick] (2.35,-5.65) -- +(0,.7);
\draw[thick] (2.35,-5.65) -- +(.72,.13);
\draw[thick] (2.18,-5.82) -- +(.01,.72);
\draw[thick] (2.18,-5.82) -- +(.75,.11);

\end{scope}
\end{tikzpicture}
\end{enumerate}
\end{thm}

\begin{proof}

We investigate each cases observing that they are exhaustive. In case (1) (a), $\mu=\lambda-\a_i$ and $\mu'=\lambda-\a_j$ for some $i\neq j$ and hence we get $\mu \wedge \mu'=\lambda-\a_i-\a_j$. Clearly $\mu \wedge \mu'$ is a cocover of both $\mu$ and $\mu'$ with no other elements in the interval.

\medskip

In case (1) (b), $\mu=\lambda-\a_K$ and $\mu'=\lambda-\a_K'$ such that $K\cup K'$ is a disconnected subdiagram of $S(A)$. Therefore we get $\mu \wedge \mu'=\lambda-\a_K-\a_{K'}$. First assume both $K$ and $K'$ are not singleton diagrams. As $\lambda \longtwoheadrightarrow \l-\a_K $, by Theorem \ref{2ndthe} (b), we obtain $(\l-\a_K)(\a_i^{\vee})=0$ for all $i\in K$. Similarly we get $(\l-\a_{K'})(\a_i^{\vee})=0$ for all $i\in K'$. Hence, $$(\lambda-\a_K-\a_{K'})(\a_i^{\vee})=0 \quad \text{for all } i\in K',$$ $$(\lambda-\a_K-\a_{K'})(\a_i^{\vee})=0 \quad \text{for all } i\in K,$$ since $K\cup K'$ is disconnected. Again by applying Theorem \ref{2ndthe} (b), we get $\l-\a_K\longtwoheadrightarrow \lambda-\a_K-\a_{K'}$ and $\lambda-\a_K'\longtwoheadrightarrow\lambda-\a_K-\a_{K'}$. If one of the two diagrams $K$ and $K'$ is a singleton diagram, then same arguments as above show that $X$ has the same diamond structure.

\medskip

In case (1) (c), first notice that none of the diagrams $K$ and $K'$ is a singleton diagram as  $K\cap K'\neq \emptyset$. As $\lambda \longtwoheadrightarrow \l-\a_K $, by Theorem \ref{2ndthe} (b), we obtain $(\l-\a_K)(\a_i^{\vee})=0$ for all $i\in K$. So, $$\l(\a_i^{\vee})=\a_K(\a_i^{\vee})=
\begin{cases}
             0  & \text{if $i$ is not an end node of } K, \\
             1  & \text{if $i$ is an end node of } K.
       \end{cases}$$
Similarly  $$\l(\a_i^{\vee})=\a_{K'}(\a_i^{\vee})=
\begin{cases}
             0  & \text{if $i$ is not an end node of } K', \\
             1  & \text{if $i$ is an end node of } K'.
       \end{cases}$$
Therefore $K\cap K'$ can contain only the common end nodes of $K$ and $K'$ and hence $|K\cap K'|\leq 2$. In this case we have $\mu \wedge \mu'=\lambda-\a_K-\a_{K'}+\sum\limits_{j\in K\cap K'}\a_{j}$. So we get, $$(\lambda-\a_K-\a_{K'}+\sum\limits_{j\in K\cap K'}\a_{j})(\a_i^{\vee})=0 \quad \text{ for all } i\in K' \setminus K,$$ $$(\lambda-\a_K-\a_{K'}+\sum\limits_{j\in K\cap K'}\a_{j})(\a_i^{\vee})=0 \quad \text{for all } i\in K \setminus K'.$$  Again by applying Theorem \ref{2ndthe} (b), we get $\l-\a_K\longtwoheadrightarrow \lambda-\a_K-\a_{K'}+\sum\limits_{j\in K\cap K'}\a_{j}$ and $\lambda-\a_K'\longtwoheadrightarrow\lambda-\a_K-\a_{K'}+\sum\limits_{j\in K\cap K'}\a_{j}$. 

\medskip

In case (2), we have $\mu=\lambda-\a_i$, $\mu'=\lambda-\a_{K'}$ such that $i\notin K'$ as $K\cap K'=\emptyset$ and $i$ is connected to some end node of $K'$ say $i_1$. In this case we get $\mu \wedge \mu'=\lambda-\a_i-\a_{K'}$. $\l \longtwoheadrightarrow \lambda-\a_{K'}$ implies $(\l-\a_{K'})(\a_i^{\vee})=0$ for all $i\in K'$ i.e. $$\l(\a_i^{\vee})=\a_{K'}(\a_i^{\vee})=
\begin{cases}
             0  & \text{if $i$ is not an end node of } K', \\
             1  & \text{if $i$ is an end node of } K'.
       \end{cases}$$ Notice that $$(\l-\a_i-\a_{i_1})(\a_{i_1}^{\vee})=1+1-2=0,$$
and $(\l-\a_i-\a_{i_1})(\a_{i}^{\vee})\geq 0$ as $\l-\a_i$ is dominant. This proves $\l-\a_i-\a_{i_1}$ is dominant. By Theorem \ref{2ndthe} (a); we get $\l-\a_i \longtwoheadrightarrow \lambda-\a_i-\a_{i_1}$. Now the fact that we have $$(\l-\a_i-\a_{K'})(\a_j^{\vee})=0 \quad \text{for all } j\in K' \text{ and } j\neq i_1,$$ together with Theorem \ref{2ndthe} (b) imply $\l-\a_i-\a_{K'}\longtwoheadrightarrow \lambda-\a_i-\a_{i_1}$. Therefore we have the following structure:

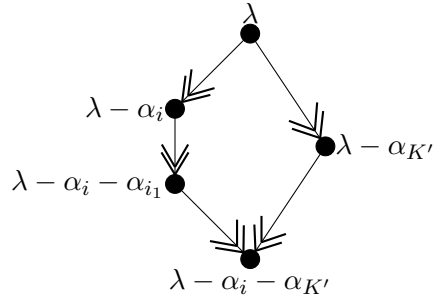
\begin{tikzpicture}[scale=.5]

  \begin{scope}[shift={(55,0)}]
  \filldraw (2,0) node[above] {$\lambda$} circle (.25cm);
  \filldraw (0,-2) node[left] {$\l-\a_i$} circle (.25cm);
\filldraw (0,-4) node[left] {$\l-\a_i-\a_{i_1}$} circle (.25cm);
  \filldraw (4,-3) node[right] {$\l-\a_{K'}$} circle (.25cm);
\filldraw (2,-6) node[below] {$\l-\a_i-\a_{K'}$} circle (.25cm);
  
  \draw (2,0) -- +(-2,-2);
\draw (0,-2) -- node[left] {$\hspace{180pt}$}+(0,-2) ;

\draw (0,-4) -- +(2,-2);
\draw (2,0) -- +(2,-3);
\draw (4,-3) -- +(-2,-3);
\draw[thick] (0.35,-1.65) -- +(0.09,.7);
  \draw[thick] (0.35,-1.65) -- +(.7,.2);
\draw[thick] (0.19,-1.82) -- +(0.09,.7);
 \draw[thick] (0.19,-1.82) -- +(.75,.19);
\draw[thick] (0,-3.55) -- +(-0.35,.6);
  \draw[thick] (0,-3.55) -- +(.35,.6);
\draw[thick] (0,-3.75) -- +(-0.36,.59);
  \draw[thick] (0,-3.75) -- +(.36,.59);
\draw[thick] (3.72,-2.58) -- +(-0.69,.26);
 \draw[thick] (3.72,-2.58) -- +(.06,.75);
\draw[thick] (3.86,-2.78) -- +(-.75,.25);
 \draw[thick] (3.86,-2.78) -- +(.08,.72);

\draw[thick] (2.3,-5.55) -- +(0,.7);
  \draw[thick] (2.3,-5.55) -- +(.65,.22);
\draw[thick] (2.14,-5.78) -- +(-0.03,.7);
 \draw[thick] (2.14,-5.78) -- +(.7,.22);
\draw[thick] (1.65,-5.65) -- +(0,.7);
  \draw[thick] (1.65,-5.65) -- +(-.72,.13);
\draw[thick] (1.82,-5.82) -- +(-.75,.12);
 \draw[thick] (1.82,-5.82) -- +(.01,.72);

\end{scope}
\end{tikzpicture}

\medskip

In case (3), we have $\mu=\lambda-\a_K$ and $\mu'=\lambda-\a_{K'}$ for some $K\cap K'=\emptyset$. Hence we get  $\mu \wedge \mu'=\lambda-\a_K-\a_{K'}$. As explained above $\l$ takes value 0 at mid nodes and value 1 at end nodes of of $K$ and $K'$. As $K\cup K'$ is connected, one end node of $K$ say, $i$ is connected to one end node of $K'$ say, $j$. Now let us consider $\l-\a_K-\a_j$. This is dominant as $\l-\a_K$ is dominant and $$(\l-\a_K-\a_j)(\a_j^{\vee})=1+1-2=0.$$
Similarly $\l-\a_{K'}-\a_i$ is also dominant. Notice that we have, $$(\l-\a_i-\a_j)(\a_j^{\vee})=1+1-2=0,$$ $$(\l-\a_i-\a_j)(\a_i^{\vee})=1-2+1=0.$$
So, $\l-\a_i-\a_j$ is also dominant. Since $(\l-\a_j)(\a_j^{\vee})=1-2<0$ and $(\l-\a_i)(\a_i^{\vee})=1-2<0$, $\l-\a_i$ and $\l-\a_j$ are not dominant. Hence we obtain $\l \longtwoheadrightarrow \l-\a_i-\a_j$. We also have, $$(\l-\a_K-\a_{K'})(\a_l^{\vee})=0 \quad \text{for all } l\in K \text{ and } l\neq i$$, as $(\l-\a_K)(\a_l^{\vee})=0$ for all $l\in K$. So by applying Theorem \ref{2ndthe} (b), we obtain $\l-\a_{K'}-\a_i\longtwoheadrightarrow \l-\a_K-\a_{K'}$. Similar arguments prove that $\l-\a_K-\a_j\longtwoheadrightarrow \l-\a_K-\a_{K'}$. Hence we have the following structure:

 \begin{tikzpicture}[scale=.65]

  \begin{scope}[shift={(55,0)}]
  \filldraw (2,0) node[above] {$\lambda$} circle (.25cm);
  \filldraw (0,-2) node[left] {$\l-\a_K$} circle (.25cm);
\filldraw (0,-4) node[left] {$\l-\a_K-\a_j$} circle (.25cm);
 \filldraw (2,-2) node[below] {}circle (.25cm);
\node[label=above:{$\l-\a_i-\a_j$} ] at (2.1,-4.7) {};
  \filldraw (4,-2) node[right] {$\l-\a_{K'}$} circle (.25cm);
\filldraw (4,-4) node[right] {$\l-\a_{K'}-\a_i$} circle (.25cm);
\filldraw (2,-6) node[below] {$\l-\a_K-\a_{K'}$} circle (.25cm);

  \draw (2,0) -- +(-2,-2);
\draw (2,0) -- +(0,-2);
\draw (2,-2) -- +(-2,-2);
\draw (2,-2) -- +(2,-2);
\draw (2,-2) -- +(0,-1.6);
\draw (0,-2) --node[left] {$\hspace{165pt}$}  +(0,-2);
\draw (0,-4) -- +(2,-2);
\draw (2,0) -- +(2,-2);
\draw (4,-2) -- +(0,-2);
\draw (4,-4) -- +(-2,-2);
  
\draw[thick] (0.35,-1.65) -- +(0.09,.7);
  \draw[thick] (0.35,-1.65) -- +(.7,.2);
\draw[thick] (0.19,-1.82) -- +(0.09,.7);
 \draw[thick] (0.19,-1.82) -- +(.75,.19);

\draw[thick] (3.65,-1.65) -- +(-.78,.29);
  \draw[thick] (3.65,-1.65) -- +(-.1,.75);
\draw[thick] (3.81,-1.82) -- +(-.8,.25);
 \draw[thick] (3.81,-1.82) -- +(-.08,.8);

\draw[thick] (0,-3.55) -- +(-0.35,.6);
  \draw[thick] (0,-3.55) -- +(.35,.6);
\draw[thick] (0,-3.75) -- +(-0.36,.59);
  \draw[thick] (0,-3.75) -- +(.36,.59);

\draw[thick] (0.65,-3.35) -- +(0.09,.7);
  \draw[thick] (0.65,-3.35) -- +(.7,.2);
\draw[thick] (0.49,-3.52) -- +(0.09,.7);
 \draw[thick] (0.49,-3.52) -- +(.75,.19);

\draw[thick] (4,-3.55) -- +(-0.35,.6);
  \draw[thick] (4,-3.55) -- +(.35,.6);
\draw[thick] (4,-3.75) -- +(-0.36,.59);
  \draw[thick] (4,-3.75) -- +(.36,.59);

\draw[thick] (2,-1.55) -- +(-0.35,.6);
  \draw[thick] (2,-1.55) -- +(.35,.6);
\draw[thick] (2,-1.75) -- +(-0.36,.59);
  \draw[thick] (2,-1.75) -- +(.36,.59);

\draw[thick] (2,-3.65) -- +(-0.35,.6);
\draw[thick] (2,-3.65) -- +(.35,.6);

\draw[thick] (3.35,-3.35) -- +(-.78,.29);
  \draw[thick] (3.35,-3.35) -- +(-.1,.75);
\draw[thick] (3.51,-3.52) -- +(-.8,.25);
 \draw[thick] (3.51,-3.52) -- +(-.08,.8);

\draw[thick] (1.65,-5.65) -- +(0,.7);
  \draw[thick] (1.65,-5.65) -- +(-.72,.13);
\draw[thick] (1.82,-5.82) -- +(-.75,.12);
 \draw[thick] (1.82,-5.82) -- +(.01,.72);

\draw[thick] (2.35,-5.65) -- +(0,.7);
\draw[thick] (2.35,-5.65) -- +(.72,.13);
\draw[thick] (2.18,-5.82) -- +(.01,.72);
\draw[thick] (2.18,-5.82) -- +(.75,.11);

\end{scope}
\end{tikzpicture}

This finishes the proof.
\end{proof}

\bigskip

\end{document}